\documentclass[reqno,final]{siamltex}
\usepackage[shortlabels]{enumitem}
\usepackage{geometry}
\usepackage{layouts}
\usepackage{stmaryrd}
\usepackage{harpoon}
\usepackage{xcolor}
\usepackage{graphicx}
\usepackage{subfigure}
\usepackage{cite}
\usepackage{diagbox}
\usepackage{mathrsfs}
\usepackage{amsmath,amssymb,amscd,amsxtra,amsfonts}

\usepackage{tikz}

\providecommand{\Dim}{\operatorname{dim}}            % dimension
\providecommand{\dim}{\Dim}
\providecommand*{\Dist}[2]{\operatorname{dist}({#1};{#2})}   % distance
\providecommand*{\Dist}[2]{\Dist{#1}{#2}}
\providecommand*{\Span}[1]{\operatorname{Span}\left\{{#1}\right\}}     % Span
\newcommand{\Bx}{{\boldsymbol{x}}}

\newcommand{\Ca}{{\cal A}}
\newcommand{\Cb}{{\cal B}}

\newcommand{\Ch}{{\cal H}}
\newcommand{\Ci}{{\cal I}}

\newcommand{\Cl}{{\cal L}}

\newcommand{\Cn}{{\cal N}}

\newcommand{\Cr}{{\cal R}}

\newcommand{\Ct}{{\cal T}}

\newcommand{\bbA}{\mathbb{A}}

\newcommand{\bbR}{\mathbb{R}}

\newcommand{\bbZ}{\mathbb{Z}}

\newcommand*{\N}[1]{\left\|{#1}\right\|}     % Double bar norm
\newcommand*{\SN}[1]{\left|{#1}\right|}      % Single bar norm
\newcommand*{\Lp}[2][\defaultdomain]{L^{#2}({#1})}

\newcommand*{\NLp}[3][\defaultdomain]{\N{#2}_{\Lp[#1]{#3}}}

\newcommand*{\Ltwo}[1][\defaultdomain]{\Lp[#1]{2}}

\newcommand*{\NLtwo}[2][\defaultdomain]{\NLp[#1]{#2}{2}}

\newcommand*{\Linf}[1][\defaultdomain]{L^{\infty}({#1})}
\newcommand*{\NLinf}[2][\defaultdomain]{\N{#2}_{\Linf[{#1}]}}

\newcommand*{\Hm}[2][\defaultdomain]{H^{#2}({#1})}

\newcommand*{\bHm}[3][\defaultdomain]{H_{#3}^{#2}({#1})}

\newcommand*{\Hone}[1][\defaultdomain]{\Hm[#1]{1}}

\newcommand*{\zbHone}[1][\defaultdomain]{\bHm[#1]{1}{0}}

\newcommand*{\NHone}[2][\defaultdomain]{{\N{#2}}_{\Hone[{#1}]}}

\newcommand*{\SNHone}[2][\defaultdomain]{{\SN{#2}}_{\Hone[{#1}]}}

\newcommand{\D}{\mathrm{d}}

\newcommand{\ol}{\overline}

\newcommand{\be}{\begin{eqnarray}}
\newcommand{\ee}{\end{eqnarray}}

\newcommand{\ben}{\begin{eqnarray*}}
\newcommand{\een}{\end{eqnarray*}}

\newtheorem{assumption}[theorem]{\sc Assumption}
\newtheorem{remark}[theorem]{\sc Remark}

\newtheorem{example}[theorem]{\sc example}

\def\address#1{\expandafter\def\expandafter\@aabuffer\expandafter
	{\@aabuffer{\affiliationfont{#1}}\relax\par
	\vspace*{13pt}}}

\newcommand{\cdotbf}{\boldsymbol{\cdot}}

\title{Finite element error estimates for the nonlinear Schr\"{o}dinger-Poisson model\thanks{LSEC, Institute of Computational Mathematics and Scientific/Engineering Computing, Academy of Mathematics and Systems Science, Chinese Academy of Sciences; School of Mathematical Science,
University of Chinese Academy of Sciences, Beijing, 100190, China. Email addresses: tcui@lsec.cc.ac.cn (T.C.), luwenhao@lsec.cc.ac.cn (W.L.), pannaiyan@lsec.cc.ac.cn (N.P.), and zwy@lsec.cc.ac.cn(W.Z.)}}

\author{Tao Cui\thanks{T. Cui was supported in part by National Key R \& D
    Program of China 2019YFA0709600 and 2019YFA0709602, China NSF grant 12171468, and Beijing Natural Science Foundation under the grant Z220003.}
  \and Wenhao Lu\thanks{W. Lu was supported in part by National Key R \& D
    Program of China 2019YFA0709600 and 2019YFA0709602.}
  \and Naiyan Pan\thanks{N. Pan was supported by China NSF grant 11831016.}
  \and Weiying Zheng\thanks{W. Zheng was supported in part by China NSF grant 11831016, National Key R \& D
    Program of China 2019YFA0709600 and 2019YFA0709602, and the National Science Fund for Distinguished Young Scholars 11725106.}}

\begin{document}
\maketitle

\begin{abstract}
In this paper, we study a priori error estimates for the finite element approximation of the nonlinear Schr\"{o}dinger-Poisson model. The electron density is defined by an infinite series over all eigenvalues of the Hamiltonian operator. To establish the error estimate, we present a unified theory of error estimates for a class of nonlinear problems. The theory is based on  three conditions: 1) the original problem has a solution $u$ which is the fixed point of a compact operator $\Ca$, 2) $\Ca$ is Fr\'{e}chet-differentiable at $u$ and $\Ci-\Ca'[u]$ has a bounded inverse  in a neighborhood of $u$, and 3) there exists an operator $\Ca_h$ which converges to $\Ca$ in the neighborhood of $u$. The theory states that $\Ca_h$ has a fixed point $u_h$ which solves the approximate problem. It also gives
the error estimate between $u$ and $u_h$, without assumptions on the well-posedness of the approximate problem. We apply the unified theory to the finite element approximation of the Schr\"{o}dinger-Poisson model and obtain optimal error estimate between the numerical solution and the exact solution. Numerical experiments are presented to verify the convergence rates of numerical solutions.
\end{abstract}



\section{Introduction}

One never-ending trend of the development of electronic components is to
reduce their sizes so that as many as possible semiconductor devices can be produced on a single chip. As the size of devices approaches the decanano length scale, quantum mechanical effects,
like the confinement in barrier structures or inversion layers as well as direct tunneling through the oxide, become prominent and must be considered in numerical simulations. Quantum-corrected macroscopic models have been studied extensively in the literature to incorporate quantum effects partially, without solving the microscopic Schr\"{o}dinger model \cite{anc87,anc89,ben98,de05,mu23,pin02}. However, for extremely small devices like quantum wire devices, quantum-mechanical methods are necessary to describe such systems so that quantum effects are taken into account adequately \cite{lan88,lau86}.

The Schr\"{o}dinger-Poisson model which incorporates all relevant quantum phenomena, is widely used in the simulation of quantum devices \cite{kai97}. The model is a nonlinear and strongly coupled system consisting of the stationary Schr\"{o}dinger equation and the Poisson equation for the electric potential. In 1993, Nier established the uniqueness and existence of the solution in \cite{nie93} and showed that the electron density operator $n: \zbHone\to H^{-1}(\Omega)$ is continuous and strictly monotone. For an electric potential $V\in\zbHone$, the electron density is defined by
\begin{align}\label{nV}
n[V]= \sum\limits_{l=1}^\infty f(\varepsilon_l-\varepsilon_F)\psi_l^2,
\end{align}
where $f>0$ stands for the thermodynamical equilibrium distribution function, $(\varepsilon_l,\psi_l)$ are eigenvalues and eigenfunctions of the Hamiltonian $\Ch_V:=-\Delta +V+V_0$ with $V_0$ being an applied potential, and $\varepsilon_F$ stands for the Fermi level.
In 1997,  Kaiser and Rehberg further improved the result of Nier for convex polyhedra or domains with $C^2$-smooth boundaries. They showed that the electron density operator $n$ is a Fr\'{e}chet-differentiable and strictly monotone operator from $\Ltwo$ to $\Ltwo$ \cite{kai97}.

Numerical methods for the Schr\"{o}dinger-Poisson model or its closely related models, Kohn-Sham and orbital-free models \cite{can12,gav07-1,gav07-2,lan10}, have been widely studied in the literature. However, works which give explicit convergence rates of numerical solutions are still seldom.
In  \cite{zha14}, Zhang, Cao, and Luo proposed a multi-scale asymptotic method for the stationary Schr\"{o}dinger-Poisson model with rapidly oscillating coefficients, where the electrons are assumed to occupy a bounded number of eigen-states $n= \sum_{l=1}^N f(\varepsilon_l-\varepsilon_F)\psi_l^2$. They also studied the finite element approximation of the homogenized model and proved that the numerical solution converges to the homogenized solution as the mesh size $h\to 0$, but the convergence rate is not obtained. In \cite{lan10}, Langwallner, Ortner, and S\"{u}li studied the linear finite element approximation to a non-convex orbital-free model, the Thomas-Fermi-von Weizs\"{a}cker (TFW) model, and obtained error estimates for numerical solutions. In \cite{can12}, Canc\'{e}s, Chakir, and Maday presented optimal convergence rate for the plane-wave discretization of the periodic TFW model. Numerical analysis for the Kohn-Sham models is more difficult than orbital-free models, since one has to deal with nonlinear eigenvalue problems in that case. Mostly,
people tend to study the case of integer occupation numbers, namely, $n= \sum_{l=1}^N \psi_l^2$.
In \cite{sur10}, the authors proved the convergence of ground state energies for finite element approximations, without giving convergence rate. In \cite{can12}, optimal a priori error estimates are established for the spectral approximation of the Kohn-Sham model under a coercivity assumption of the second-order optimality condition.
In \cite{che13}, the authors presented a priori error estimates for both ground state energies and numerical solutions.

The purpose of this paper is to establish a priori error estimates for finite element approximation of the
Schr\"{o}dinger-Poisson model in which the electron density is defined with a series of infinitely many eigen-states \eqref{nV}. To our best knowledge, explicit convergence rates are still unavailable in the literature for finite element approximations of the nonlinear eigenvalue problem. Generally,
the approximation of a nonlinear problem needs to establish three kinds of theories: the existence and local uniqueness of the exact solution $u$, the existence of the approximate solution $u_h$, and the estimation of the error $u-u_h$. The existence of $u_h$ is not implied by the existence of $u$ and could be more difficult than the latter. In some cases, the existence and local uniqueness of the approximate solution are treated as assumptions to estimate the approximation errors (see e.g. \cite{hao22}). The first objective of this paper is to present a unified theory of error estimates for nonlinear problems. The theory is established upon three conditions:
\begin{enumerate}[leftmargin=5mm]
\item The nonlinear problem has a solution $u$ which is the fixed point of a compact operator $\Ca$.

\item $\Ca$ is Fr\'{e}chet-differentiable at $u$ and $\Ci-\Ca'[u]$ has a bounded inverse in a neighborhood of $u$.

\item The approximate problem can be formulated into an operator equation $\Ca_hu_h = u_h$ where $\Ca_h$ converges to $\Ca$ uniformly in a neighborhood of $u$.
\end{enumerate}
The theory gives the existence of $u_h$ and a priori estimate of $u-u_h$ in terms of the approximation parameter $h$. The most attractive property of the theory is that no assumptions are made about the well-posedness of the approximate problem.

The second objective is to apply the unified theory to the finite element approximation of the
nonlinear Schr\"{o}dinger-Poisson model. The discrete problem is proposed upon two kinds of approximations:
\begin{enumerate}[leftmargin=5mm]
\item truncating the series \eqref{nV} into the sum of a finite number of terms, and

\item numerical discretization of the coupled problem.
\end{enumerate}
Therefore, the total error between the exact solution and the approximate solution is the addition of both the truncation error and the discretization error. Choosing the number of truncated terms $L_h=O(|\ln h|^{3/2})$, we show that the two kinds of errors are of the same order. Furthermore, we establish optimal error estimate for the numerical solution in the $H^1$ norm. In the last section, we present two numerical experiments to verify the convergence rates of numerical solutions.

The paper is organized as follows. In section 2, we present the unified theory of error estimates for nonlinear problems. In section 3, we introduce the Schr\"{o}dinger-Poisson model and verify two conditions of the abstract theorem. In section~4, we propose a finite element approximation to the Schr\"{o}dinger-Poisson model and prove some error estimates for eigenvalues and eigenfunctions. In section~5, we establish a priori error estimates for numerical solutions by using the unified theory. In section~6, we present two numerical experiments to verify the convergence rates of numerical solutions. In section~7, we draw some conclusions from the paper.

We shall adopt some notations of Banach spaces and their subsets.
\begin{itemize}[leftmargin=5mm]
\item $X$ and $Y$ denote two Banach spaces with $Y\subset X$.

\item $X'$ denotes the dual space of $X$. Particularly, $X'=H^{-1}(\Omega)$ if $X=\zbHone$.

\item $\Cl(X,Y)$ denotes the space of all linear and continuous operators from $X$ to $Y$, particularly, $\Cl(X):=\Cl(X,X)$.

\item $B(u,r)=\left\{w\in X: \N{v-w}_X \le r\right\}$ denotes the closed ball of $X$ (or $\zbHone$), where the radius of the ball is $r>0$ and the center is $u\in X$ (or $u\in\zbHone$).

\item $\Cb_r=\left\{w\in Y: \N{w}_Y < r\right\}$ denotes the open ball of $Y$ (or $\Ltwo$), where the radius of the ball is $r>0$ and the center is $0$.

\item $\Ltwo$ denotes the space of square integrable functions on $\Omega$. The inner product on $\Ltwo$ is denoted by $(u,v)$.

\item $H^m(\Omega)$ denotes the Sobolev space whose functions have square-integrable partial derivatives up to order $m\ge 0$, and $H^m_0(\Omega)\subset H^m(\Omega)$ denotes the subspace of functions with vanishing traces on $\partial\Omega$. The inner product of $\Hone$ is denoted by $(u,v)_1$.

\item $C_r>0$ denotes the generic constant which depends only on $\Omega$ and $r$, and $C>0$ denotes the generic constant which depends only on $\Omega$.
\end{itemize}

\section{A unified theory of error estimates for nonlinear problems}
\label{sec:framwork}

The purpose of this section is to prove an abstract theorem of error estimates for general nonlinear problems. Throughout this section, we assume that $X$ is compactly embedded into $Y$.
Suppose $F$: $X\to X'$ is a continuous operator and that the nonlinear operator equation
\begin{align}\label{eq:F}
F[u]=0 \quad \hbox{on}\;\; X,
\end{align}
has a solution $u\in X$.
Since we are interested in the approximation and error analysis for problem \eqref{eq:F}, it is reasonable to make the assumption.
\vspace{1mm}

\begin{assumption} \label{ass-1}
There exists a continuous operator $\Ca$: $Y\to X$ such that the solution $u$ is a fixed point of $\Ca$, namely,
\begin{align}\label{eq:A}
\Ca[u]=u\quad \hbox{in}\;\; X.
\end{align}
\end{assumption}

The assumption is rather mild since for most nonlinear problems, the existence of solution can be proven by means of fixed-point iterations.

Let $X_h$ be a subspace of $X$ and let $\Ca_h$: $Y\to X_h$ be a continuous operator which is an approximation of $\Ca$. We consider the approximate equation
\begin{align}\label{eq:Ah}
\Ca_h[u_h] =u_h \quad \hbox{in}\;\; X_h.
\end{align}
The purpose is to prove that the nonlinear problem \eqref{eq:Ah} possesses at least one solution, providing that $\Ca_h$ is sufficiently ``close'' to $\Ca$ as $h\to h_\infty$, where $h_\infty\in\bbR$ or $h_\infty=\pm\infty$. An error estimate between $u$ and $u_h$ will be presented. The proof is based on a quasi-Newton map.

\begin{lemma}\label{lem:compact}
Suppose $A$ is a bounded and closed subset of $X$. Let ${\rm cl}_YA$ denote the closure of $A$ in the norm of $Y$. Then ${\rm cl}_YA$ is compact in $Y$, and ${\rm cl}_YA$ is convex if $A$ is convex.
\end{lemma}
\begin{proof}
Let $\{v_n\}\subset {\rm cl}_Y A$ be a bounded sequence. It suffices to extract a subsequence which converges in ${\rm cl}_Y A$. For each $v_n$, there exists a sequence $\{v_{n,k}\}\subset A$ such that
\ben
\lim_{k\to \infty} \big\|v_{n,k}-v_n\big\|_Y=0.
\een
For any $n\geq 1$, there exists an integer $N(n)$ large enough such that
$\big\|v_{n,N(n)}-v_n\big\|_Y<1/n$.
Since $\{v_{n,N(n)}\}\subset A$, there exists a subsequence $\{v_{n_k,N(n_k)}\}$ which converges to some $v\in {\rm cl}_YA$ in the norm of $Y$.
It follows that
\ben
\lim_{k\to \infty} \|v_{n_k}-v\|_Y\leq \lim_{k\to \infty}\|v_{n_k}-v_{n_k,N(n_k)}\|_Y
+\lim_{k\to \infty}\|v_{n_k,N(n_k)}-v\|_Y =0.
\een
This means that ${\rm cl}_YA$ is compact in $Y$.

To prove the convexity of ${\rm cl}_Y A$, we pick $v,w \in {\rm cl}_Y A$ and $\{v_{n}\},\{w_{n}\}\subset A$ which satisfy
\ben
\lim_{n\to \infty} \big(\|v_{n}-v\|_Y+ \|w_{n}-w\|_Y\big) =0.
\een
For any $\lambda\in (0,1)$, $\lambda v_{n}+(1-\lambda)w_{n} \in A$ implies
$\lambda v+(1-\lambda)w \in {\rm cl}_Y A$. The proof is finished.
\end{proof}
\vspace{1mm}

\begin{theorem} \label{thm:compact}
Let $\Ci$ be the identical operator on $X$.
Suppose the three conditions below hold.
\begin{itemize}[leftmargin=5mm]
\item {\bf Condition~1.} $\Ca$ is Fr\'{e}chet-differentiable at $u$ and the Fr\'{e}chet derivative $\Ca'[u]:Y\to X$ is a bounded and linear operator defined via the limit
\begin{align}\label{IDA-1}
\lim_{\N{v-u}_Y\to 0} \frac{\N{\Ca[v]-\Ca[u]-\Ca'[u](v-u)}_X}{\N{v-u}_Y} =0.
\end{align}

\item {\bf Condition~2.} There exists an $r>0$ such that $\Ci-\Ca'[u]$ restricted on $B(u,r)$ has a bounded inverse, namely, there exist two constants $M_0$ and $M_1$ such that
\begin{align}\label{IDA}
\sup_{0\ne v\in B(u,r)}\frac{\N{\big(\Ci-\Ca'[u]\big)v}_X}{\N{v}_X}\leq M_0,\qquad
\sup_{0\ne v\in B(u,r)}\frac{\N{v}_X}{\N{\big(\Ci-\Ca'[u]\big)v}_X}
\leq M_1.
\end{align}

\item {\bf Condition~3.} $\Ca_h$ admits the local approximation property
\begin{align}\label{epsh}
\lim_{h\to h_\infty} \epsilon(h) =0, \quad \hbox{where}\;\;
\epsilon(h):=\sup_{v\in B(u,r)}{\N{\Ca_h[v]-\Ca[v]}_X}.
\end{align}
\end{itemize}
There exist an $r_0\in\big(0,(1+M_0M_1)^{-1}r\big]$ and an $h_0\in\bbR$ such that,
for any $h$ between $h_0$ and $h_\infty$,
problem~\eqref{eq:Ah} has a solution $u_h\in B(u,r_0)$ which admits the error estimates
\begin{align}\label{abs-err}
\N{u_h-u}_X \le 2M_1\N{\Ca_h[u_h]-\Ca[u_h]}_X
\le (1+2M_0M_1)\N{u_h-u}_X.
\end{align}
\end{theorem}
\begin{proof}
First we define the residual operator for the first-order Taylor's formula at $u$
\begin{align}\label{defR}
\Cr_u[v] = \big(u-\Ca[u]\big)-\big(v-\Ca[v]\big)-\big(\Ci-\Ca'[u]\big)(u-v) \qquad
\forall\,v\in B(u,r).
\end{align}
Since $u-\Ca[u]$ is Fr\'{e}chet-differentiable at $u$, \eqref{IDA-1} implies
\begin{align}\label{limR}
\lim_{\N{v-u}_Y\to 0}\frac{\N{\Cr_u[v]}_X}{\N{u-v}_Y} =0.
\end{align}
There exists an $r_0 <(1+M_0M_1)^{-1}r$ such that
\begin{align}\label{estR}
\N{\Cr_u[v]}_X \le \N{u-v}_Y/(2M_1) \qquad \forall\, v\in B(u,r_0).
\end{align}
By \eqref{epsh}, there exists an $h_0\in\bbR$ such that, for any $h$ between $h_0$ and $h_\infty$,
\begin{align}\label{estA-Ah}
\sup_{v\in B(u,r)}\N{\Ca_h[v]-\Ca[v]}_X \le r_0 /(2M_1).
\end{align}

Next we define an operator $\Cn$ on ${\rm cl}_Y B(u,r_0)$ as follows
\begin{equation}\label{opr-N}
\Cn[v]= \big(\Ci-\Ca'[u]\big)^{-1}\Ca_h[v] -\big(\Ci-\Ca'[u]\big)^{-1}\Ca'[u](v).
\end{equation}
For any $v\in B(u,r_0)$, by $u=\Ca[u]$ and direct calculations, we find that
\begin{align}
\Cn[v]-u=\,& v-u-\big(\Ci-\Ca'[u]\big)^{-1}\big(v-\Ca_h[v]\big)\notag \\
=\,& \big(\Ci-\Ca'[u]\big)^{-1}\big\{\big(\Ci-\Ca'[u]\big)(v-u)-v+\Ca_h[v]+u-\Ca[u]\big\} \notag \\
=\,& \big(\Ci-\Ca'[u]\big)^{-1}\big\{R_u[v]+\Ca_h[v]-\Ca[v]\big\}. \label{Nv-u}
\end{align}
Applying \eqref{estR} and \eqref{estA-Ah} to the third equality of \eqref{Nv-u} yields
\begin{align*}
\N{\Cn[v]-u}_X \le\, M_1 \big\{\N{R_u[v]}_X+\N{\Ca_h[v]-\Ca[v]}_X\big\} \le r_0
\qquad \forall\,v\in B(u,r_0).
\end{align*}
The inequality implies that $\Cn$ maps $B(u,r_0)$ continuously into itself.

Next we prove that $\Cn$ also maps ${\rm cl}_Y B(u,r_0)$ into ${\rm cl}_Y B(u,r_0)$. For any $v\in {\rm cl}_Y B(u,r_0)$, there is a sequence $\{v_n\}\subset B(u,r_0)$ which converges to $v$ in the norm of $Y$. Then $\{\Cn[v_n]\}\subset B(u,r_0)$. Since both $\Ca_h$ and $\Ca'[u]$ are continuous operators, we have
\begin{align*}
\lim_{n\to \infty}\N{\Cn[v_n]-\Cn[v]}_Y
\le M_1\lim_{n\to \infty}
\N{\Ca_h[v_n]-\Ca_h[v] - \Ca'[u](v_n-v)}_X =0.
\end{align*}
This implies $\Cn[v]\in {\rm cl}_Y B(u,r_0)$. Therefore, $\Cn$ maps ${\rm cl}_Y B(u,r_0)$ into itself.
Since $X$ is compactly embedded into $Y$, by Lemma~\ref{lem:compact}, ${\rm cl}_Y B(u,r_0)$ is a nonempty, convex, and compact subset of $Y$.
By Schauder's fixed point theorem (cf. \cite[Page 502, Theorem~3]{eva98}), there exists a $u_h\in {\rm cl}_Y B(u,r_0)$ satisfying $u_h=\Cn[u_h]$.
Then \eqref{opr-N} implies $\Ca_h[u_h]=u_h$. We conclude the existence of solution to the approximate problem \eqref{eq:Ah}.

It is left to prove \eqref{abs-err}. From \eqref{Nv-u} and \eqref{estR}, we have
\begin{align*}
\N{u_h-u}_X =\,& \N{\Cn[u_h]-u}_X
= \N{\big(\Ci-\Ca'[u]\big)^{-1}\big(R_u[u_h]+\Ca_h[u_h]-\Ca[u_h]\big)}_X \notag \\
\le\,& \frac{1}{2}\N{u-u_h}_X +M_1\N{\Ca_h[u_h]-\Ca[u_h]}_X .
\end{align*}
This shows the lower bound of the error estimate. From \eqref{Nv-u} and \eqref{estR}, we also have
\begin{align*}
\N{\Ca_h[u_h]-\Ca[u_h]}_X \le  \N{\big(\Ci-\Ca'[u]\big)(u_h-u)-R_u[u_h]}_X
\le \big\{M_0+(2M_1)^{-1}\big\}\N{u_h-u}_X.
\end{align*}
The proof is finished.
\end{proof}
\vspace{1mm}

%
%\begin{corollary} \label{cor:abs-error}
%Suppose the three conditions below hold.
%\begin{enumerate}[leftmargin=5mm]
%\item [1)] $\Ca$ is Fr\'{e}chet-differentiable at $u$ and $\Ca'[u]:Y\to X$ is a bounded operator.
%
%\item [2)] The operator $\Ci-\Ca'[u]: X\to X$ $B(u,r)$ has a bounded inverse.
%
%
%\item [3)] $\lim\limits_{h\to 0} \sup\limits_{v\in B(u,r)}{\N{\Ca_h[v]-\Ca[v]}_X}=0$.
%\end{enumerate}
%There exist an $r_0\in\big(0,(1+M_0M_1)^{-1}r\big]$ and an $h_0>0$ such that
%problem~\eqref{eq:Ah} has a solution $u_h\in B(u,r_0)$ for any $h\in (0, h_0]$.
%Moreover, $u_h$ admits the error estimates
%\begin{align}\label{eq:abs-err}
%\N{u_h-u}_X \le 2M_1\N{\Ca_h[u_h]-\Ca[u_h]}_X
%\le (1+2M_0M_1)\N{u_h-u}_X.
%\end{align}
%\end{corollary}

The previous theorem gives the error estimate in the norm of $X$. We can easily obtain the error estimate with respect to the weaker norm $\|\cdot\|_Y$.

\begin{corollary} \label{cor:compact-low-norm}
Suppose the conditions 1) and 3) in Theorem~\ref{thm:compact} hold, and moreover, there exist two positive constants $M_0'$ and $M_1'$ such that
\begin{align}\label{M0M1-1}
\sup_{v\in B(u,r)}\frac{\N{\big(\Ci-\Ca'[u]\big)v}_Y}{\N{v}_Y}\leq M_0',\qquad
\sup_{v\in B(u,r)}\frac{\N{v}_Y}{\N{\big(\Ci-\Ca'[u]\big) v}_Y}
\leq M_1'.
\end{align}
Then $u_h$ admits the error estimates
\begin{align}\label{abs-err-low-norm}
\N{u_h-u}_Y \le 2M_1'\N{\Ca_h[u_h]-\Ca[u_h]}_Y
\le (1+2M_0'M_1')\N{u_h-u}_Y.
\end{align}
\end{corollary}

\begin{proof}
Given the constant $M_1'$, we can choose the constant $M_1$ in \eqref{IDA} large enough such that $M_1\ge M_1'$ without loss of generality.
For any $h$ between $h_0$ and $h_\infty$ and any $v\in B(u,r_0)$, using \eqref{epsh} and \eqref{limR}, we have
\begin{align*}
\N{\Cr_u[v]}_Y \le \N{u-v}_Y/(2M_1'),\qquad
\N{\Ca_h[v]-\Ca[v]}_Y \le r_0/(2M_1') .
\end{align*}
Let $\Cr_u$ be the residual operator defined in \eqref{defR}.
By \eqref{M0M1-1} and arguments similar to the proof of Theorem~\ref{thm:compact}, we have
\begin{align*}
\N{u_h-u}_Y =\,& \big\|\big(\Ci-\Ca'[u]\big)^{-1}\big(R_u[u_h]+\Ca_h[u_h]-\Ca[u_h]\big)\big\|_Y
\le  M_1'\N{R_u[u_h]+\Ca_h[u_h]-\Ca[u_h]}_Y.
\end{align*}
We get $\N{u_h-u}_Y \le 2M_1'\N{\Ca_h[u_h]-\Ca[u_h]}_Y$. The second inequality of \eqref{abs-err-low-norm} can be proven similarly as in the proof of Theorem~\ref{thm:compact}.
\end{proof}

\section{The Schr\"{o}dinger-Poisson model}

Next we apply the abstract theory in Theorem~\ref{thm:compact} to the nonlinear Schr\"{o}dinger-Poisson model and establish a priori error estimates for the Galerkin finite element approximation of the problem. Suppose $\Omega$ is a bounded convex polyhedron in $\bbR^3$.

\subsection{The model problem}

The Schr\"{o}dinger-Poisson model is the coupling of the Poisson equation of the electric potential and the stationary Schr\"{o}dinger equation.
After proper nondimensionalization, the Poisson equation has the form
\begin{equation}\label{eq:Poisson}
-\Delta V = n[V]- n_D\quad{\rm in}\;\; \Omega,\qquad
V=0\quad \hbox{on}\;\; \partial \Omega,
\end{equation}
where $V$ denotes the electrostatic potential, $n_D\in L^2(\Omega)$ is the doping profile, and $n:\Ltwo\to\Ltwo$ is the electron density operator. Suppose $V_0\in\Linf$ is an applied potential.
The Hamiltonian of the system is defined as $\Ch_V  = -\Delta +V+V_0$.
The eigenvalue problem of the Schr\"{o}dinger equation has the form
\begin{equation}\label{eq:eigen}
\Ch_V \psi_l=\varepsilon_l\psi_l \quad {\rm in}\;\; \Omega,\qquad
\psi_l=0 \quad \hbox{on}\;\; \partial \Omega,\qquad
l=1,2,\cdots,
\end{equation}
where $\varepsilon_l,\psi_l$ are the eigenvalues and eigenfunctions of $\Ch_V$, respectively, satisfying
\ben
\varepsilon_1<\varepsilon_2\le \cdots\le \varepsilon_l\le \cdots, \qquad
\NLtwo{\psi_l} =1.
\een

For clearness, we consider two typical forms of the thermodynamical equilibrium distribution function in this paper, the Boltzmann distribution and the Fermi-Dirac distribution,
\begin{equation}\label{eq:f}
f(t)= f_0e^{-\mu t}, \qquad
f(t)= f_0/(1+e^{\mu t}), \qquad f_0 >0.
\end{equation}
Clearly $f\in C^\infty(\bbR)$ and satisfies assumption~(d) of \cite[Section~1]{kai97}. With slight modifications, our theory can be extend to more general forms of $f$ subject to this assumption. We write $\varepsilon_l=\varepsilon_l[V]$ to specify the dependency of $\varepsilon_l$ on $V$.
The electron density is defined by
\begin{equation}\label{eq:density}
n[V]=\sum^{\infty}_{l=1}f(\varepsilon_l[V]-\varepsilon_F[V])\SN{\psi_l}^2.
\end{equation}
Here $\varepsilon_F=\varepsilon_F[V]$ denotes the Fermi level which depends on $V+V_0$. Given the total number of electrons $N_0$, $\varepsilon_F$ is determined through the conservation of electrons (see \cite{kai97})
\begin{equation}\label{eq:epsF}
N_0 =\sum^{\infty}_{l=1}f(\varepsilon_l[V]-\varepsilon_F[V]).
\end{equation}
In this way, \eqref{eq:Poisson}--\eqref{eq:epsF} forms a nonlinear and strongly coupled system.

\begin{lemma}
Suppose the geometric multiplicity of the eigenvalue $\varepsilon_l$ is $q$ and define the eigenspace associated with $\varepsilon_l$ as $X_l= \big\{w\in\zbHone: \Ch_V w =\varepsilon_l w\big\}$. Then the definition of $n[V]$ does not depend on the choice of an orthonormal basis of $X_l$.
\end{lemma}
\begin{proof}
Let $\{\psi_{l+j}: j=0,\cdots,q-1\}$ be an orthonormal basis of $X_l$, satisfying $(\psi_{l+i},\psi_{l+j})_1=\delta_{ij}$. For notational convenience, we omit the dependency on $V$ and rewrite \eqref{eq:density} as follows
\begin{equation*}
n=n_l +\sum_{k=1}^{l-1}f(\varepsilon_k-\varepsilon_F) \psi_k^2
+\sum^{\infty}_{k=l+q}f(\varepsilon_k-\varepsilon_F)\psi_k^2,\qquad
n_l:=f(\varepsilon_{l}-\varepsilon_F)\sum^{q-1}_{i=0}\psi_{l+i}^2.
\end{equation*}
Let $\{w_{l+j}: j=0,\cdots,q-1\}$ be another orthonormal basis
of $X_{l}$. Each $\psi_{l+i}$ can be represented as
\ben
\psi_{l+i} = \sum_{k=0}^{q-1} \alpha_{i,k} w_{l +k},\qquad
\sum_{k=0}^{q-1}\alpha_{i,k}\alpha_{j,k}
=(\psi_{l+i},\psi_{l+j})_1= \delta_{ij}.
\een
This indicates that $\bbA =(\alpha_{i,k})_{i,k=0}^{q-1}$ is a unitary matrix. It is clear that
\begin{equation*}
\sum^{q-1}_{i=0}\psi_{l+i}^2
=\sum^{q-1}_{i=0}\bigg|\sum^{q-1}_{k=0}\alpha_{i,k} w_{l +k}\bigg|^2
=\sum^{q-1}_{k=0} w_{l +k}^2.
\end{equation*}
Therefore, neither $n_l$ nor $n$ depends on the choice of the orthonormal basis of $X_l$.
\end{proof}

Each eigenvalue of the Schr\"{o}dinger equation can be viewed as an operator from $\Ltwo$ to $\bbR$.
The next lemma is useful and shows that the operator $\varepsilon_l:\Ltwo\to\bbR$ is continuous.

\begin{lemma}\label{lem:cont-eigen}
Suppose $u,w\in\Ltwo$ and $\varepsilon_l[u]$, $\varepsilon_l[w]$ are the eigenvalues of $\Ch_u$ and $\Ch_w$, respectively. Then for any $l\ge 1$,
\ben
\lim\limits_{\NLtwo{w-u}\to 0} \SN{\varepsilon_l[w]-\varepsilon_l[u]}= 0 .
\een
\end{lemma}
\begin{proof}
Let $\lambda_1,\lambda_2,\cdots$ be the eigenvalues of $(-\Delta)|_{\zbHone}$.
Recall \cite[equation (2.1)]{kai97} (see also \cite{li83,ree78}) that the spectrum of $(-\Delta)|_{\zbHone}$ is discrete and has a countable number of eigenvalues
\begin{align}\label{asym-lam}
0<\lambda_1\le \lambda_2\le \cdots\le\lambda_l\le \cdots,\qquad
\lim_{l\to\infty}(l^{-2/3}\lambda_l) = C_\Omega>0.
\end{align}
For any $w\in\Ltwo$, using the Cauchy-Schwarz inequality and Poincar\'{e}'s inequality, we have
\begin{equation}\label{uv2}
\SN{(w,v^2)} \le \N{w}_{L^2(\Omega)}
    \N{v}_{L^6(\Omega)}^{3/2} \N{v}_{L^2(\Omega)}^{1/2}
\le \frac{1}{2}\SNHone{v}^2 + C \NLtwo{w}^4\NLtwo{v}^2.
\end{equation}

We endow $\varepsilon_l[u]$ an eigenfunction $\psi_l$ satisfying $\Ch_u\psi_l=\varepsilon_l[u]\psi_l$. If $\varepsilon_l[u]$ has multiplicity $q$, say,  $\varepsilon_{l}[u]=\cdots =\varepsilon_{l+q-1}[u]$ without loss of generality, we choose the eigenfunctions satisfying $(\psi_{l+i},\psi_{l+j})=\delta_{ij}$ for $0\le i,j\le q-1$. Define $W_l(\Ch_u)=\Span{\psi_1,\cdots,\psi_{l+q-1}}$. By \eqref{uv2} and
the minimum-maximum principle, we have
\begin{align}\label{ub-lam}
\lambda_l \le \max_{v\in W_l(\Ch_u)} \frac{\SNHone{v}^2}{\NLtwo{v}^2}
= \max_{v\in W_l(\Ch_u)} \frac{a(u;v,v)-(u+V_0,v^2)}{\NLtwo{v}^2}
\le 2\varepsilon_l[u]+C \NLtwo{u+V_0}^4.
\end{align}
Similarly, we have
\begin{align}\label{lb-lam}
\varepsilon_l[u]\le 1.5\lambda_l +C\NLtwo{u+V_0}^4.
\end{align}

Suppose $w\to u$ strongly in the $L^2$-norm. It is easy to see that
\begin{align*}
\SN{a(u;v,v)-a(w;v,v)} =\SN{(u-w, v^2)}
\le \NLtwo{u-w}\N{v}_{L^4(\Omega)}^2\to 0 .
\end{align*}
Since $W_l(\Ch_u)$ is finite-dimensional, the minimum-maximum principle shows
\begin{align*}
\varlimsup_{w\to u}\varepsilon_l[w]
    \le\lim_{w\to u} \max_{v\in W_l(\Ch_u)} \frac{a(w;v,v)}{\NLtwo{v}^2}
    =\max_{v\in W_l(\Ch_u)}\lim_{w\to u} \frac{a(w;v,v)}{\NLtwo{v}^2}
    =\varepsilon_l[u].
\end{align*}
Moreover, by Poincar\'{e}'s inequality and arguments similar to \eqref{ub-lam}, we have
\begin{align*}
\max_{v\in W_l(\Ch_w) } \frac{\N{v}_{L^4(\Omega)}^2}{\NLtwo{v}^2} \le\,&
\max_{v\in W_l(\Ch_w) } \frac{C\SNHone{v}^2}{\NLtwo{v}^2} \le
C\big(\varepsilon_l[w]+ \NLtwo{w+V_0}^4\big).
\end{align*}
The right-hand side is bounded by a constant $C_u$ depending only on $\varepsilon_l[u]$ and $\NLtwo{u+V_0}$.
Then the Cauchy-Schwarz inequality implies
\begin{align*}
\varepsilon_l[u]\le\,& \max_{v\in W_l(\Ch_w) } \frac{a(u;v,v)}{(v,v)}
= \max_{v\in W_l(\Ch_w) } \frac{a(w;v,v)+(u-w,v^2)}{(v,v)}
\le \varepsilon_l[w]+ C_u\NLtwo{u-w}.
\end{align*}
This implies $\varepsilon_l[u]\le \varliminf\limits_{w\to u}\varepsilon_l[w]$. Therefore, we conclude
$\varepsilon_l[u] = \lim\limits_{w\to u}\varepsilon_l[w]$.
\end{proof}

\subsection{An equivalent operator equation}
\label{sec:opr}

Now we rewrite the Schr\"{o}dinger-Poisson system into a nonlinear elliptic problem
\begin{equation}\label{prob-V}
-\Delta V - n[V] = -n_D\quad \hbox{in}\;\; \Omega,\qquad
V=0\quad \hbox{on}\;\;\partial \Omega.
\end{equation}
The weak formulation is to seek $V\in\zbHone$ such that
\begin{equation}\label{weak-V}
(\nabla V,\nabla v) - (n[V],v) = -(n_D,v)\qquad \forall\,v\in\zbHone.
\end{equation}
In \cite{kai97,nie93}, the authors prove that problem \eqref{weak-V} has a unique solution $V\in\zbHone$.
Since $\Omega$ is a bounded convex polyhedron,
by \cite[Theorem~7.2]{kai97} and  \cite[Corollary~7.1]{kai97}, the solution satisfies
\ben
V\in H^2(\Omega)\cap\zbHone,\qquad V\in L^\infty(\Omega).
\een
Moreover, by \cite[Proposition~6.1]{kai97}, the electron density $n$ provides a map from $\Ltwo$ to $C(\bar\Omega)$, namely, $n[v]\in C(\bar\Omega)$ for any $v\in \Ltwo$.
%Therefore, $\Cp=-\Delta -n$ is a nonlinear operator from $\zbHone$ to $H^{-1}(\Omega)$. Problem \eqref{prob-V} is further equivalent to the operator equation
%\begin{equation}\label{opr-V}
%\Cp[V] +n_D =0 .
%\end{equation}

\subsection{Conditions 1 and 2 of Theorem~\ref{thm:compact}}
\label{sec:dA}

In order to apply the abstract theory in Theorem~\ref{thm:compact} to the Schr\"{o}dinger-Poisson model, we define $X=H_0^1(\Omega)$ and $Y=L^2(\Omega)$.
Clearly $X$ is compactly embedded into $Y$. To write the nonlinear equation \eqref{prob-V} into a fixed-point problem, we should first define an operator $\Ca:Y\to X$ as given in \eqref{eq:A}.

For $u\in L^2(\Omega)$, $\Ca[u]:=\xi$ is the solution to the Poisson equation
\begin{align}\label{Poisson-xi}
-\Delta \xi =n[u]-n_D\quad \hbox{in}\;\; \Omega,\qquad
\xi =0 \quad\hbox{on}\;\; \partial \Omega.
\end{align}
Clearly the fixed point of $\Ca$, satisfying $V=\Ca[V]$, is the solution to the nonlinear equation \eqref{prob-V}. By \cite[Theorem~6.3]{kai97}, the electron-density operator
$n:\Ltwo\to\Ltwo$  is Fr\'{e}chet-differentiable at any $u\in L^2(\Omega)$.
Its Fr\'{e}chet-derivative $n'[u]$ provides a linear and continuous operator from $\Ltwo$ to $\Ltwo$, that is, $n'[u]\in \Cl(L^2(\Omega))$. It follows that
\begin{align}\label{cont-A}
\lim_{\NLtwo{v-u}\to 0} \NHone{\Ca[v]-\Ca[u]}
\le \lim_{\NLtwo{v-u}\to 0} \NLtwo{n[v]-n[u]} =0.
\end{align}
Therefore, $\Ca$ is a continuous operator and satisfies Assumption~\ref{ass-1}.
The regularity theory of elliptic equations shows $\Ca[u]\in H^2(\Omega)\cap\zbHone$.

Next we show that $\Ca$ satisfies the conditions 1 and 2 of Theorem~\ref{thm:compact}.
For $w\in\Ltwo$, let $\eta$ be the unique solution to the elliptic problem
\begin{align}\label{Poisson-eta}
-\Delta\eta = n'[u](w)\quad \hbox{in}\;\; \Omega,\qquad
\eta=0\quad\hbox{on}\;\; \partial \Omega.
\end{align}
Clearly $\Ca'[u](w):=\eta$ defines an operator $\Ca'[u]:\Ltwo\to \zbHone$ which satisfies
\ben
\N{\Ca'[u]}_{\Cl(L^2(\Omega),\zbHone)}\le \N{n'[u]}_{\Cl(L^2(\Omega))}.
\een
As $\NLtwo{v}\to 0$, from \eqref{Poisson-xi} and \eqref{Poisson-eta}, we also have
\begin{align*}
\frac{\N{\Ca[v+u]-\Ca[u]-\Ca'[u](v)}_{H^1(\Omega)}}{\NLtwo{v}}
\le \frac{\NLtwo{n[v+u]-n[u]-n'[u](v)}}{\NLtwo{v}} \to 0.
\end{align*}
Therefore, $\Ca'[u]$ is the Fr\'{e}chet-derivative of $\Ca$ at $u$.
\vspace{1mm}

\begin{lemma} \label{lem:condition-2}
There exists a constant $C>0$ depending only on $\Omega$ such that
\ben
\N{(\Ci-\Ca'[u])^{-1}}_{\Cl(\zbHone)} \le C\qquad
\forall\, u\in\Ltwo.
\een
\end{lemma}
\begin{proof}
By \cite[Theorem~6.5]{kai97}, the electron-density operator $n$ is monotone, namely,
\begin{align*}
(n[u_2]-n[u_1],u_1-u_2)\geq 0\qquad \forall\, u_1,u_2\in L^2(\Omega).
\end{align*}
This implies $(n'[u]w,w) \le 0$ for all $w\in L^2(\Omega)$. The compactness of the injection $H^1_0(\Omega)\hookrightarrow\Ltwo $ shows that the operator $\Ca'[u]:H^1_0(\Omega)\to H^1_0(\Omega)$ is compact. Since
\begin{align*}
\big(\nabla (\Ci-\Ca'[u])w,\nabla w\big)=(\nabla w,\nabla w)-(n'[u]w,w)\ge \SNHone{w}^2,
\end{align*}
the operator $(\Ci-\Ca'[u])^{-1}:\zbHone\to\zbHone$ exists and satisfies
$\SNHone{(\Ci-\Ca'[u])^{-1}v} \le \SNHone{v}$ for any $v\in\zbHone$.
The proof is finished upon using Poincar\'{e}'s inequality.
\end{proof}

\section{Finite element approximation}
\label{sec:fem}

The purpose of this section is to propose a finite element method for solving the Schr\"{o}dinger-Poisson model. The series on the right-hand side of \eqref{eq:density} must be truncated in practical computations.
For any $M>0$, we introduce a cutoff function which satisfies
\begin{equation}\label{eq:chiM}
\chi_M\in C^\infty(\bbR),\qquad
\chi_M(t)=
\begin{cases}
1 &\hbox{if}\;\; t\leq M,\\
0 &\hbox{if}\;\; t>M+1,
\end{cases}
\quad\hbox{and}\quad \chi_M'\le 0\;\;\; \hbox{in}\;\; (M,M+1).
\end{equation}
The truncated distribution function is defined as $f_M=\chi_M f$.

\subsection{The discrete problem}

Let $\Ct_h$ be a quasi-uniform, shape-regular, and tetrahedral mesh of $\Omega$. The diameter of an element $K\in\Ct_h$ is denoted by $h_K$. The mesh size of $\Ct_h$ is defined by $h:=\max\limits_{K\in\Ct_h}h_K$. The linear continuous finite element space is defined by
\begin{align*}
X_h = \big\{v_h\in\zbHone: v_h|_K\in P_1(K),\;\forall\, K\in\Ct_h\big\},\qquad
N_h :=\dim (X_h),
\end{align*}
where $P_1(K)$ is the space of linear polynomials on $K$.

Given $u\in\Ltwo$, we define the bilinear form $a(u;\cdotbf,\cdotbf): \Hone\times\Hone\to \bbR$ by
\begin{equation}\label{eq:aform}
a(u;\psi,\varphi):= (\nabla \psi,\nabla \varphi) +((u+V_0)\psi, \varphi)
\qquad \forall\, \psi,\varphi\in\Hone.
\end{equation}
The eigenpair $(\varepsilon_l[u],\psi_l)\in\bbR\times\zbHone$ of the Hamiltonian $\Ch_u$ satisfies the weak formulation
\begin{equation}\label{eq:eigen-u}
a(u;\psi_{l},\varphi\big) =\varepsilon_{l}[u] \big(\psi_{l},\varphi\big)
\quad \forall\,\varphi\in \zbHone,\qquad \NLtwo{\psi_{l}} =1.
\end{equation}
The finite element approximation of \eqref{eq:eigen-u} is to seek
$(\varepsilon_{l,h}[u],\psi_{l,h})\in\bbR\times X_h$ such that
\begin{equation}\label{eq:eigenh}
a(u;\psi_{l,h},\varphi_h\big) =\varepsilon_{l,h}[u]\big(\psi_{l,h},\varphi_h\big)
\quad \forall\,\varphi_h\in X_h,\qquad \NLtwo{\psi_{l,h}} =1.
\end{equation}
The eigenvalues satisfy $\varepsilon_{1,h}[u]\le \varepsilon_{2,h}[u] \le \cdots\le \varepsilon_{N_h,h}[u]$.

The finite element approximation to problem \eqref{eq:Poisson} is to seek $V_h\in X_h$ such that
\begin{align}\label{eq:prob-Vh}
(\nabla V_h,\nabla v_h) = (n_h[V_h]-n_D, v_h)\qquad \forall\, v_h\in X_h.
\end{align}
For any $u\in\Ltwo$, $n_h[u]\in C(\bar\Omega)$ is defined by
\begin{equation}\label{eq:nh}
n_h[u] =\sum^{L_h[u]}_{l=1}f_{M_h}(\varepsilon_{l,h}[u]-\varepsilon_{F,h}[u])\psi_{l,h}^2,
\qquad M_h=2\mu^{-1}|\ln h|,
\end{equation}
where $L_h=L_h[u]$ is the smallest integer satisfying $f_{M_h}(\varepsilon_{L_h,h}[u]-\varepsilon_{F,h}[u]) =0$. The discrete Fermi level $\varepsilon_{F,h}[u]$ is determined through the conservation of electrons
\begin{equation}\label{eq:epsLh}
N_0 =\sum^{L_h}_{l=1}f_{M_h}(\varepsilon_{l,h}[u]-\varepsilon_{F,h}[u]).
\end{equation}
Without causing confusions, we enlarge the numbers of discrete eigenvalues and discrete eigenfunctions by defining $\varepsilon_{l,h}[u] :=\varepsilon_{N_h,h}[u]$ and $\psi_{l,h} :=\psi_{N_h,h}$ for all $l>N_h$. Then $f_{M_h}(\varepsilon_{l,h}[u]-\varepsilon_{F,h}[u])=0$ for all $l\ge L_h$ and
\begin{equation*}
n_h[u] =\sum^{\infty}_{l=1}f_{M_h}(\varepsilon_{l,h}[u]-\varepsilon_{F,h}[u])\psi_{l,h}^2,
\qquad
N_0 =\sum^{\infty}_{l=1}f_{M_h}(\varepsilon_{l,h}[u]-\varepsilon_{F,h}[u]).
\end{equation*}

\begin{lemma}\label{lem:cont-epslh}
Suppose $\varepsilon_{l,h}[u]$ and $\varepsilon_{l,h}[w]$ are the eigenvalues of \eqref{eq:eigenh} with potentials $u$ and $w$, respectively. Then for any $1\le l\le L_h[u]$,
\begin{align}\label{cont-eigen-h}
\lim\limits_{\NLtwo{w-u}\to 0}\SN{\varepsilon_{l,h}[w]-\varepsilon_{l,h}[u]}=0.
\end{align}
\end{lemma}
\begin{proof}
Since $X_h$ is a finite-dimensional space, the proof is similar to the proof of Lemma~\ref{lem:cont-eigen}, and is easier than the latter.
We omit the details here.
\end{proof}

\subsection{Error estimates for approximate eigenvalues and eigenfunctions}

Throughout this section, for a given potential $u\in\Ltwo$, we abbreviate the notations and write
\ben
\varepsilon_{l}=\varepsilon_{l}[u],\quad \varepsilon_{F}=\varepsilon_{F}[u],\quad
\varepsilon_{l,h}=\varepsilon_{l,h}[u],\quad \varepsilon_{F,h}=\varepsilon_{F,h}[u],\quad
L_h=L_h[u].
\een
In view of \eqref{eq:density}--\eqref{eq:epsF} and \eqref{eq:nh}--\eqref{eq:epsLh}, we should estimate the errors
$\varepsilon_{l}  -\varepsilon_{l,h}$, $\psi_{l}-\psi_{l,h}$, and $\varepsilon_{F} -\varepsilon_{F,h}$.
It suffices to study the linear Schr\"{o}dinger equation \eqref{eq:eigen-u} and its finite element approximation \eqref{eq:eigenh}.
Let $\varepsilon_l$ and $\varepsilon_{l,h}$ be the eigenvalues of \eqref{eq:eigen-u} and \eqref{eq:eigenh} by counting multiplicity, respectively. Suppose $\psi_l$ is an eigenfunction of $\Ch_u$ associated with $\varepsilon_l$.

\begin{lemma}\label{lem:eigenf}
There exists a constant $C>0$ depending only on $\Omega$ such that
\begin{align}
&\SNHone{\psi_l}^2 \le 2|\varepsilon_l|+ C \NLtwo{u+V_0}^2,
    \label{psi-H1-ub}\\
&\SNHone{\psi_l}^2 \ge \max\big(0.5|\varepsilon_l|-C\NLtwo{u+V_0}^2,C\big),
    \label{psi-H1-lb}\\
&\N{\psi_l}_{H^2(\Omega)} \le C\big(|\varepsilon_l|+
\NLtwo{u+V_0}^2\NHone{\psi_l}\big).  \label{psi-H2}
\end{align}
\end{lemma}
\begin{proof}
Using \eqref{eq:eigen-u} and the Schwarz inequality, we have
\begin{align*}
\SNHone{\psi_l}^2 \le C \NLtwo{u+V_0}\N{\psi_l}^{3/2}_{L^6(\Omega)}+\SN{\varepsilon_l}
\le C \NLtwo{u+V_0}\SNHone{\psi_l}^{3/2} +\SN{\varepsilon_l}  .
\end{align*}
Similarly, the lower bound estimate is given by
\begin{align*}
\SNHone{\psi_l}^2 \ge 0.5|\varepsilon_l| - C \NLtwo{u+V_0}^4.
\end{align*}
Moreover, Poincar\'{e}'s inequality implies $\SNHone{\psi_l}\ge C\NLtwo{\psi_l} = C$.
So \eqref{psi-H1-lb} holds.

Finally, using the $H^2$-regularity of eigenfunctions, we have
\ben
\N{\psi_l}_{H^2(\Omega)} \le C(\NLtwo{u+V_0}\NLinf{\psi_l} +\SN{\varepsilon_l}).
\een
Then \eqref{psi-H2} follows from the interpolation inequality $\N{\psi_l}_{L^\infty(\Omega)} \le C\N{\psi_l}_{H^2(\Omega)}^{1/2} \NHone{\psi_l}^{1/2}$.
\end{proof}

\begin{lemma}\label{lem:err-eigen}
Suppose $u\in\Cb_r$ and $h$ is small enough. For each discrete eigenfunction $\psi_{l,h}$, there exist an eigenfunction $\psi_l$ associated with $\varepsilon_l$ and a constant $C_r$ depending only on $r,\Omega$ such that
\begin{align*}
\NHone{\psi_l-\psi_{l,h}} \le C_r |\varepsilon_l|h , \qquad
|\varepsilon_l-\varepsilon_{l,h}|\le C_r \varepsilon_l^2 h^2.
\end{align*}
\end{lemma}
\begin{proof}
By the interpolation inequality and Poincar\'{e}'s inequality, we have
\begin{align}\label{vL4}
\N{v}_{L^4(\Omega)}^2\le C_0\SN{v}_{H^1(\Omega)}^{3/2}\N{v}_{L^2(\Omega)}^{1/2}\qquad
\forall\,v\in\zbHone.
\end{align}
Setting $\phi=u+C_0^4\NLtwo{u+V_0}^4$, we can rewrite
\eqref{eq:eigen-u} and \eqref{eq:eigenh} as follows
\begin{align}
& a(\phi;\psi_l,\varphi)=\hat{\varepsilon}_l\big(\psi_l,\varphi\big)
\qquad \forall\,\varphi\in\zbHone,     \label{lam-v}\\
& a(\phi;\psi_{l,h}, \varphi_h) =\hat{\varepsilon}_{l,h}(\psi_{l,h},\varphi_h)
\qquad \forall\,\varphi_h\in X_h,\label{lam-vh}
\end{align}
where $\hat{\varepsilon}_l=\varepsilon_l+C_0^4\NLtwo{u+V_0}^4$ and $\hat{\varepsilon}_{l,h}=\varepsilon_{l,h}+C_0^4\NLtwo{u+V_0}^4$.

Now we define a weighted norm
\ben
\N{v}_u:= \big(\SNHone{v}^2 + C_0^4\NLtwo{u+V_0}^4\NLtwo{v}^2\big)^{1/2}
\qquad \forall\,v\in\Hone.
\een
An application of Young's inequality shows that
\begin{align*}
a(\phi;v,v) \ge\,& \N{v}_u^2 -\N{u+V_0}_{L^2(\Omega)}\NHone{v}^{3/2}\NLtwo{v}^{1/2}
\ge \frac{1}{4} \N{v}_u^2,\\
a(\phi;v,w) \le\,& C\N{v}_u\N{w}_u.
\end{align*}
Therefore, $a(\phi;\cdotbf,\cdotbf)$ provides a coercive and continuous bilinear form on $\Hone$.

Suppose $l\ge 1$ is fixed and $h$ is small enough. By \cite[(8.47b)--(8.47c)]{bab91},
there exists a constant $C>0$ independent of $h$ and $\varepsilon_l$ such that
\begin{align*}
&\NHone{\psi_l-\psi_{l,h}} \le C \inf_{v_h\in X_h}\N{\psi_l-v_h}_u
\le C \big(h+h^2\NLtwo{u+V_0}^2\big) \SN{\psi_l}_{H^2(\Omega)}, \\
&|\hat{\varepsilon}_l-\hat{\varepsilon}_{l,h}| \le
C \hat{\varepsilon}_l\NHone{\psi_l}^{-2}
\inf_{v_h\in X_h}\N{\psi_l-v_h}_u^2
\le C \big(h^2 +h^4\NLtwo{u+V_0}^4\big)
\hat{\varepsilon}_l \SN{\psi_l}^2_{H^2(\Omega)}\NHone{\psi_l}^{-2}.
\end{align*}
The proof is finished by using Lemma~\ref{lem:eigenf}.
\end{proof}

\subsection{Error estimate for the approximate Fermi level}

Next we estimate the error $\varepsilon_F-\varepsilon_{F,h}$.
The lemma below states that the truncation number $L_h$ grows in the rate $\SN{\ln h}^{3/2}$ as $h\to 0$.
\vspace{1mm}

\begin{lemma}\label{lem:Lh}
Suppose $u\in \Cb_r$ with $r>0$ and $h$ is small enough. There exist a constant $C_r>0$ depending only on $r,\Omega$ and two positive constants $C_0$ and $C_1$ depending only on $\Omega$ such that
\ben
C_0\SN{\ln h} -C_r\le L_h^{2/3}\le C_1\SN{\ln h} +C_r.
\een
\end{lemma}
\begin{proof}
Notice that $\varepsilon_{L_h-1,h} \le M_h+1+\varepsilon_{F,h}\le \varepsilon_{L_h,h}$.
By \eqref{asym-lam} and \eqref{ub-lam}, we have
\begin{align}\label{Lh2/3-ub}
L_h^{2/3}\le C_0 \lambda_{L_h-1}\le C_r+ C\varepsilon_{L_h-1}
    \le C_r+ C\varepsilon_{L_h-1,h}
\le C |\ln h|+  C_r\varepsilon_{F,h}.
\end{align}
Since $f$ is strictly decreasing, it has an inverse function $f^{-1}$.
Then $f(\varepsilon_{1,h}-\varepsilon_{F,h})\le N_0$ implies
\ben
\varepsilon_{F}\le \varepsilon_{F,h}\le \varepsilon_{1,h}-f^{-1}(N_0).
\een
Let $\phi_1$ be the eigenfunction associated with the first eigenvalue $\lambda_1$ of $(-\Delta)|_{\zbHone}$ and satisfying $\NLtwo{\phi_1}=1$. Let $\phi_{h}$ be the $L^2$ projection of $\phi_1$ onto $X_h$. From \cite{bra01}, we have
\ben
\NLtwo{\phi_1-\phi_{h}}\le Ch\NHone{\phi_1},\qquad
\NHone{\phi_h}\le C\NHone{\phi_1}.
\een
As $h\to 0$, these indicate
\begin{align*}
\varepsilon_{1,h} = \min_{v_h\in X_h}\frac{a(u;v_h,v_h)}{\NLtwo{v_h}^2}
\le \frac{a(u;\phi_h,\phi_h)}{\NLtwo{\phi_h}^2}
\le C \frac{(1+\NLtwo{u})\NHone{\phi_1}^2}{\NLtwo{\phi_1}^2-h^2\NHone{\phi_1}^2}
\le C (1+\NLtwo{u}).
\end{align*}
We conclude $|\varepsilon_{F,h}|\le C_r$. From \eqref{Lh2/3-ub}, we have
$L_h^{2/3}\le C |\ln h| +C_r$.

Similarly, by \eqref{asym-lam}, \eqref{ub-lam}, and Lemma~\ref{lem:err-eigen}, we find that
\begin{align*}
L_h^{2/3} \ge C\varepsilon_{L_h} -C_r
\ge C\varepsilon_{L_h,h}-C_r h^2 \varepsilon_{L_h}^2-C_r
\ge C|\ln h| -C_r h^2L_h^{4/3} -C_r.
\end{align*}
Since $L_h^{2/3}\le C |\ln h|+C_r$, we obtain $L_h^{2/3} \ge C_0|\ln h| -C_r$ when $h$ is small enough.
\end{proof}
\vspace{1mm}

\begin{lemma}\label{lem:fermi-h}
There exists an $h_0>0$ depending only on $r$ such that, for any $0<h\le h_0$,
the mapping $\varepsilon_{F,h}: \Cb_r\to\bbR$ is continuous and injective.
\end{lemma}
\begin{proof}
We define $g_h(y):=\sum^{L_h}\limits_{l=1}f_{M_h}(\varepsilon_{l,h}-y)$ for any $y\in\mathbb{R}$. Clearly $g$ is a continuous and increasing function and satisfies $\lim\limits_{y\to -\infty}g_h(y)=0$. Moreover,
by Lemma~\ref{lem:Lh}, there exists an $h_0>0$ depending only on $r$ such that $L_h>N_0$ for any $h\in (0,h_0]$. This implies
$\lim\limits_{y\to \infty}g_{h}(y) >N_0$.
By the intermediate value theorem, there exists a $y_0\in\bbR$ satisfying
$g(y_0)=N_0$.

Since $f_{M_h}(\varepsilon_{l,h}-y_0)\ge f_{M_h}(\varepsilon_{l+1,h}-y_0)$, we infer that $f_{M_h}(\varepsilon_{1,h}-y_0)>0$. Thanks to \eqref{eq:f}, $f_{M_h}$ is strictly monotone in the neighborhood of $\varepsilon_{1,h}-y_0$. Then $g(y)$ is strictly monotone in the neighborhood of $y_0$. Therefore, $y_0$ is unique. So $\varepsilon_{F,h}[u]:=y_0$ defines an injective mapping from $\Ltwo$ to $\bbR$.
The continuity of $\varepsilon_{F,h}$ follows directly from Lemma~\ref{lem:cont-epslh} and the continuity of $f_{M_h}$.
\end{proof}
\vspace{1mm}

Now we are ready to prove the error estimate between the exact Fermi level $\varepsilon_F$ and the discrete Fermi level $\varepsilon_{F,h}$ .
In view of \eqref{eq:nh}, we have used the truncated distribution function $f_{M_h}$ with $M_h=2\mu^{-1}|\ln h|$. Since $f(t)$ decreases exponentially with $t$, it is expected that the truncation error and the numerical error are of the same order.

\begin{theorem}\label{thm:err-Fh}
Suppose $u\in\Cb_r$ with $r>0$.
There exists an $h_0>0$ depending only on $r$ and $\Omega$ such that
$0\le \varepsilon_{F,h}-\varepsilon_F \le C_r h^2$ holds for any $h\in (0,h_0]$.
\end{theorem}
\begin{proof} The Galerkin approximation of the eigenvalue problem implies $\varepsilon_{l,h}\ge \varepsilon_l$. In view that
\begin{align*}
\sum^{L_h}_{l=1}f_{M_h}(\varepsilon_{l,h}-\varepsilon_{F})
\le \sum^{L_h}_{l=1}f_{M_h}(\varepsilon_{l}-\varepsilon_{F})
=\sum^{\infty}_{l=1}f(\varepsilon_l-\varepsilon_F)
= N_0 =\sum^{L_h}_{l=1}f_{M_h}(\varepsilon_{l,h}-\varepsilon_{F,h}),
\end{align*}
we easily know $\varepsilon_{F,h}\ge \varepsilon_{F}$. Since $f_{M_h}(\varepsilon_{l,h}-\varepsilon_{F,h})=0$ for all $l\ge L_h$, it is clear that
\begin{align*}
\sum^{\infty}_{l=1}\big[f_{M_h}(\varepsilon_{l,h}-\varepsilon_{F,h})-f_{M_h}(\varepsilon_{l}-\varepsilon_{F})\big]
\le \sum_{l=1}^\infty\big|f(\varepsilon_l-\varepsilon_F)
    -f_{M_h}(\varepsilon_l-\varepsilon_F)\big|
\le C e^{-\mu M_h} \le Ch^2.
\end{align*}

By the mean value theorem, there is a $\xi_l$ between $\varepsilon_{l,h}-\varepsilon_{F,h}$ and $\varepsilon_l-\varepsilon_{F}$ such that
\begin{equation}\label{eq:xil}
\sum_{l=1}^{\infty}\big[f_{M_h}(\varepsilon_{l,h}-\varepsilon_{F,h}) -f_{M_h}(\varepsilon_l-\varepsilon_{F})\big]
=\sum_{l=1}^{\infty} f_{M_h}'(\xi_l) \big[(\varepsilon_{l,h}-\varepsilon_l)
+(\varepsilon_{F}-\varepsilon_{F,h})\big].
\end{equation}
From \eqref{eq:f}, it is easy to see $|f'(t)|\ge \mu f(t)$. Using Lemma~\ref{lem:err-eigen} and $\varepsilon_{F,h}\ge \varepsilon_{F}$, we deduce that
\ben
f(\varepsilon_{1,h}-\varepsilon_{F,h})\ge
f(\varepsilon_{1}-\varepsilon_{F})e^{-C_r\varepsilon_1^2 h^2}.
\een
By \eqref{eq:epsF} and \eqref{ub-lam}--\eqref{lb-lam}, both $\varepsilon_1$ and $\varepsilon_F$ are bounded by a constant depending only on $r$. Then
\ben
\sum\limits_{l=1}^{\infty} |f_{M_h}'(\xi_l)| \ge \mu
\max\big( f(\varepsilon_{1,h}-\varepsilon_{F,h}), f(\varepsilon_{1}-\varepsilon_{F})\big)
\ge C_r.
\een
Finally, from Lemmas~\ref{lem:eigenf} and \ref{lem:err-eigen}, we deduce that
\begin{align*}
\varepsilon_{F,h}-\varepsilon_{F} \le C h^2
+ \sum_{l=1}^{\infty} |f_{M_h}'(\xi_l)|(\varepsilon_{l,h}-\varepsilon_l)
\le C h^2 + C_r h^2 \sum_{l=1}^{\infty} (1+ \varepsilon_l^2)|f_{M_h}'(\xi_l)|.
\end{align*}
The proof is finished upon the fact that $|f'(t)|$ decreases exponentially as $t\to\infty$.
\end{proof}

\section{A priori error estimates}

The purpose of this section is to establish the error estimate between the exact solution $V$ and the numerical solution $V_h$ by using the abstract result in Theorem~\ref{thm:compact}.
Remember that conditions 1 and 2 are already verified in section~\ref{sec:dA}. The key step is to construct a discrete operator $\Ca_h:\Ltwo\to X_h$ which satisfies condition~3 of
Theorem~\ref{thm:compact}.

For $u\in L^2(\Omega)$, let $\Ca_h[u]:=\xi_h$ be the solution to the discrete problem
\begin{align}\label{Poisson-xih}
(\nabla \xi_h,\nabla v_h) = (n_h[u]-n_D, v_h)\qquad \forall\, v_h\in X_h.
\end{align}
Clearly the fixed point of $\Ca_h$, which satisfies $V_h=\Ca_h[V_h]$, is the solution to problem \eqref{eq:prob-Vh}.
Next we prove that $\Ca_h$ maps the open ball $\Cb_r$ continuously to $X_h$ when $h$ is small enough.
\vspace{1mm}

\begin{theorem}\label{thm:cont-Ah}
There exists an $h_0>0$ depending only on $r$ such that, for any $h\in (0, h_0]$, the mapping $\Ca_{h}: \Cb_r\to X_h$ is continuous.
\end{theorem}
\begin{proof}
For any $u\in\Cb_r$, let $\Ch_{h}[u]\in\Cl(X_h)$ be the discrete Schr\"{o}dinger operator defined as follows: for $\phi_h\in X_h$, $\Ch_h[u]\phi_h\in X_h$ solves the discrete problem
\ben
(\Ch_h[u]\phi_h,v_h)_1= a(u;\phi_{h},v_h\big)\qquad \forall\,v_h\in X_h.
\een
Suppose $\varepsilon_{l,h}[u]$ is an eigenvalue of $\Ch_h[u]$ and define $X_{l,h}[u]:=\ker(\varepsilon_{l,h}[u]\Ci-\Ch_h[u])$.
Consider the orthogonal decomposition $X_h = X_{l,h}[u]\oplus X_{l,h}^\perp[u]$ where $X_{l,h}^\perp[u]$ denotes the orthogonal complement space of $X_{l,h}[u]$.
The closed graph theorem implies $(\varepsilon_{l,h}[u]\Ci-\Ch_h[u])^{-1}\in \Cl(Y_{l,h}[u], X_{l,h}^\perp[u])$, where
$Y_{l,h}[u]:=\mathrm{Range}(\varepsilon_{l,h}[u]\Ci-\Ch_h[u])$.

Similarly, for $w\in\Ltwo$, suppose $\psi_{l,h}$ is an eigenfunction of $\Ch_h[w]$ belonging to $\varepsilon_{l,h}[w]$ and satisfying $\NLtwo{\psi_{l,h}}=1$. Consider the splitting
\ben
\psi_{l,h}= \phi_{l,h}+\phi_{l,h}^\perp,\qquad
\phi_{l,h}\in X_{l,h}[u],\quad
\phi_{l,h}^\perp\in X_{l,h}^\perp[u].
\een
Then $\phi_{l,h}$ is an eigenfunction of $\Ch_h[u]$ belonging to $\varepsilon_{l,h}[u]$ and satisfying
$\NHone{\phi_{l,h}}\le\NHone{\psi_{l,h}}$.
Clearly $v_h^\perp :=(\varepsilon_{l,h}[u]\Ci-\Ch_h[u])^{-1}\phi_{l,h}^\perp\in X_{l,h}^\perp[u]$ satisfies
\begin{align*}
\lim_{w\to u}\NLtwo{\phi_{l,h}^\perp-(\varepsilon_{l,h}[w]\Ci-\Ch_h[w])v_h^\perp}
=\,&\lim_{w\to u}\NLtwo{(\varepsilon_{l,h}[u]-\varepsilon_{l,h}[w]+\Ch_h[w]-\Ch_h[u])v_h^\perp} \\
=\,& \lim_{w\to u}\NLtwo{(u-w)v_h^\perp} \\
=\,&  0.
\end{align*}
Notice $\phi_{l,h}\in X_{l,h}[u]$, $\psi_{l,h}\in X_{l,h}[w]$, and that
both $\Ch_h[u],\Ch_h[w]$ are self-adjoint, it is clear that
\begin{align*}
&\big(\phi_{l,h}, (\varepsilon_{l,h}[u]\Ci-\Ch_h[u])v_h^\perp\big) =\big((\varepsilon_{l,h}[u]\Ci-\Ch_h[u])\phi_{l,h},v_h^\perp\big)=0, \\
&\lim_{w\to u} \NLtwo{\phi_{l,h}^\perp}^2
= \lim_{w\to u}\big(\psi_{l,h}, \phi_{l,h}^\perp\big)
= \lim_{w\to u}\big(\psi_{l,h},(\varepsilon_{l,h}[w]\Ci-\Ch_h[w])v_h^\perp\big) =0.
\end{align*}
Define $\hat\phi_{l,h}:=\phi_{l,h}/\NLtwo{\phi_{l,h}}$. We know from
Lemmas~\ref{lem:cont-eigen} and \ref{lem:fermi-h} that
\begin{align}\label{fMh-wu}
\lim_{w\to u} \big\|f(\varepsilon_{l,h}[w]-\varepsilon_{F,h}[w])\psi_{l,h}^2
-f(\varepsilon_{l,h}[u]-\varepsilon_{F,h}[u])\hat\phi_{l,h}^2\big\|_{\Ltwo} =0 .
\end{align}

In view of \eqref{eq:nh}, $n_h[w]$ and $n_h[u]$ are given by
\ben
n_h[w] =\sum^{\infty}_{l=1}f_{M_h}(\varepsilon_{l,h}[w]-\varepsilon_{F,h}[w])\psi_{l,h}^2,
\quad
n_h[u] =\sum^{\infty}_{l=1}f_{M_h}(\varepsilon_{l,h}[u]-\varepsilon_{F,h}[u])
\hat\phi_{l,h}^2.
\een
They are actually sums of finitely many terms. Then \eqref{Poisson-xih} and \eqref{fMh-wu} show that
\ben
\lim_{w\to u} \NHone{\Ca_h[w]-\Ca_h[u]} \le
\lim_{w\to u} \NLtwo{n_h[w]-n_h[u]} =0.
\een
The proof is finished.
\end{proof}

\begin{lemma}\label{lem:n-nh}
Suppose $u\in\Cb_r$ with $r>0$ and $h$ is small enough.
There is a constant $C_r$ depending only on $r$ such that
\ben
\NHone{n[u]-n_h[u]} \le C_r h.
\een
\end{lemma}
\begin{proof}
Suppose $\psi_l$ and $\psi_{l,h}$ are the eigenfunctions belonging to $\varepsilon_{l}=\varepsilon_l[u]$ and $\varepsilon_{l,h}=\varepsilon_{l,h}[u]$, respectively, and satisfy $\NHone{\psi_l-\psi_{l,h}} \le C_r|\varepsilon_l| h$ according to Lemma~\ref{lem:err-eigen}. Then we have
\begin{align}\label{I123}
n[u] -n_h[u] = \sum^{\infty}_{l=1}f(\varepsilon_l-\varepsilon_{F}) \psi_l^2
    -\sum^{\infty}_{l=1}f_{M_h}(\varepsilon_{l,h}-\varepsilon_{F,h})\psi_{l,h}^2
    = I_1 + I_2 + I_3,
\end{align}
where
\begin{align*}
I_1 =\,& \sum^{\infty}_{l=1}\big[f(\varepsilon_l-\varepsilon_{F})
    -f_{M_h}(\varepsilon_l-\varepsilon_{F})\big] \psi_l^2,\\
I_2 =\,& \sum^{\infty}_{l=1}\big[f_{M_h}(\varepsilon_l-\varepsilon_{F})
    -f_{M_h}(\varepsilon_{l,h}-\varepsilon_{F,h})\big] \psi_l^2, \\
I_3 =\,& \sum^{\infty}_{l=1}f_{M_h}(\varepsilon_{l,h}-\varepsilon_{F,h})
    \big(\psi_l^2-\psi_{l,h}^2\big).
\end{align*}

Using equation \eqref{eq:chiM} and Lemma~\ref{lem:eigenf}, we immediately get
\begin{align*}
\NHone{I_1}\le C_r e^{-\mu M_h}\le C_r h^2.
\end{align*}
Let $\xi_l$ be given in \eqref{eq:xil}.
By Lemma~\ref{lem:eigenf} and arguments similar to the proof of Theorem~\ref{thm:err-Fh}, we have
\begin{align*}
\NHone{I_2} \le \sum_{l=1}^{\infty} \big|f'_{M_h}(\xi_l)
    (\varepsilon_{l,h}-\varepsilon_{l} + \varepsilon_{F}-\varepsilon_{F,h})\big|
    \NHone{\psi_l^2}
\le C_r h^2 \sum_{l=1}^{\infty}\big|f'_{M_h}(\xi_l)\big|\varepsilon_l^4
\le C_r h^2.
\end{align*}
Let $\pi_h\psi_l\in X_h$ be the nodal interpolation of $\psi_l$. It is well-known that
\begin{align}\label{err-pi}
\N{\psi_l-\pi_h\psi_l}_{H^j(\Omega)}\le C h^{2-j}\SN{\psi_l}_{H^2(\Omega)}
\le C_r(1+|\varepsilon_l|)h^{2-j},\qquad j=0,1.
\end{align}
Since $e_{l,h}:=\pi_h\psi_l-\psi_{l,h}\in C(\ol\Omega)$, there exists an $\Bx_0\in\Omega$ satisfying
$\SN{e_{l,h}(\Bx_0)} = \N{e_{l,h}}_{C(\ol\Omega)}$.
Let $L_0$ be a line segment connecting $\Bx_0$ and $\partial\Omega$. Since $e_{l,h}$ vanishes on $\partial\Omega$, it is clear that
\begin{align*}
\SN{e_{l,h}(\Bx_0)} \le \int_{L_0}\SN{\nabla e_{l,h}}\D l
\le \sum_{K\in\Ct_h} h^{-1/2}\SNHone[K\cap L_0]{e_{l,h}}
\le Ch^{-1}\SNHone{e_{l,h}}
\le C \N{\psi_l}_{H^2(\Omega)},
\end{align*}
where we have used \eqref{err-pi} and Lemma~\ref{lem:err-eigen} in the last inequality.
This indicates
\begin{align*}
\NLinf{\psi_{l,h}}\le \NLinf{\pi_h\psi_l} +\NLinf{e_{l,h}}
\le C \N{\psi_l}_{H^2(\Omega)}.
\end{align*}
By Lemmas~\ref{lem:eigenf} and \ref{lem:err-eigen}, there exists a constant depending only on $\Omega$ such that
\begin{align*}
\NHone{\psi_l^2-\psi_{l,h}^2} \le C
\big(\N{\psi_l}_{H^2(\Omega)}+\N{\psi_{l,h}}_{L^\infty(\Omega)}\big) \NHone{\psi_l-\psi_{l,h}}
\le C_r h \N{\psi_l}_{H^2(\Omega)}^2 .
\end{align*}
It follows that
\begin{align*}
\NHone{I_3} \le \sum_{l=1}^{\infty} f_{M_h}(\varepsilon_{l,h}-\varepsilon_{F,h})
    \NHone{\psi_l^2-\psi_{l,h}^2}
\le C_r h.
\end{align*}
The proof is finished by inserting the three inequalities into \eqref{I123}.
\end{proof}
\vspace{1mm}

Now we are ready to present the main result of this section.

\begin{theorem}\label{thm:V-Vh}
There exist an $r_0>0$ depending only $\N{\Ci-\Ca'[V]}_{\Cl(\Ltwo,\zbHone)}$ and an $h_0>0$ depending only on $r_0$ and $\NLtwo{V}$ such that, for any $h\in (0,h_0]$,
\begin{itemize}[leftmargin=6mm]
\item problem \eqref{eq:prob-Vh} has a solution $V_h\in B(V,r_0)$,

\item $\NHone{V -V_h} \le C_V h$, where $C_V$ is a constant depending only on $\NLtwo{V}$.
\end{itemize}
\end{theorem}
\begin{proof}
Take $r=2\NLtwo{V}$ and $u\in \Cb_r$. From \eqref{Poisson-xi}, there exists a constant $C>0$ depending only on $\Omega$ such that
\ben
\N{\Ca[u]}_{H^2(\Omega)} \le C\NLtwo{n[u]-n_D}\le C_r,
\een
where we have used the stability $\NLtwo{n[u]}\le C_r$ in the second inequality (see \cite[Lemma~6.2]{kai97}).

Remember that $\xi=\Ca[u]$ is the solution to problem \eqref{Poisson-xi} and $\xi_h=\Ca_h[u]$ is the solution to problem \eqref{Poisson-xih}. Moreover, let $\zeta$ be the solution to the elliptic problem
\begin{align*}
-\Delta \zeta =n_h[u]-n_D\quad \hbox{in}\;\; \Omega,\qquad
\zeta =0 \quad\hbox{on}\;\; \partial \Omega.
\end{align*}
Standard error estimates for Galerkin finite element approximations show that
\begin{align*}
\NHone{\zeta-\xi_h} \le Ch\N{\zeta}_{H^2(\Omega)} \le Ch \NLtwo{n_h[u]-n_D}.
\end{align*}
By Lemma~\ref{lem:n-nh}, we also have
\begin{align*}
\NHone{\xi-\zeta} \le C\NLtwo{n[u] -n_h[u]} \le C_r h.
\end{align*}
Combining the two inequalities yields
\begin{align}\label{est-A-Ah}
\NHone{\Ca[u]-\Ca_h[u]} = \NHone{(\xi-\zeta)+(\zeta-\xi_h)}
\le C_r h +Ch \NLtwo{n[u]-n_D} \le C_r h.
\end{align}
This implies $\lim\limits_{h\to 0}\NHone{\Ca[u]-\Ca_h[u]} =0$.

By Lemma~\ref{lem:condition-2}, the three conditions of Theorem~\ref{thm:compact} hold.
There exist an $r_0\in (0,r]$ depending only on $\N{\Ci-\Ca'[V]}_{\Cl(\Ltwo,\zbHone)}$ and an $h_0>0$ depending only on $r_0$ and $\NLtwo{V}$ such that problem \eqref{eq:prob-Vh} has a solution $V_h\in B(V,r_0)\cap X_h$.
Then $\NLtwo{V_h}\le r_0 +\NLtwo{V}\le 2r$.
Moreover, there exists a constant $C>0$ depends only on $r_0$ and $\NLtwo{V_h}$ such that
\ben
\NHone{V-V_h} \le C \NHone{\Ca_h[V_h]-\Ca[V_h]} \le C h.
\een

we obtain $\NHone{V-V_h} \le C_V h$ where the constant $C_V$ depends only on $r_0$ and $\NLtwo{V}$.

\end{proof}
\vspace{1mm}

\begin{corollary}\label{cor:n-nh}
Suppose $V$ and $V_h$ are the solutions to problem \eqref{prob-V} and \eqref{eq:prob-Vh}, respectively, as given in Theorem~\ref{thm:V-Vh}. Let $n[V]$ be the electron density defined in \eqref{eq:density} and $n_h[V_h]$ be the discrete electron density defined in \eqref{eq:nh}. There exists a constant $C_V$ depending only on $\NLtwo{V}$ and $\N{\Ci-\Ca'[V]}_{\Cl(\Ltwo,\zbHone)}$ such that
\begin{align*}
\NLtwo{n[V] -n_h[V_h]} \le C_V h.
\end{align*}
\end{corollary}
\begin{proof}
From Theorem~\ref{thm:V-Vh}, we know that
$V_h\in B(V,r_0)\subset \Cb_r$ where $r=r_0+\NLtwo{V}$ and $r_0>0$ depends only on $\N{\Ci-\Ca'[V]}_{\Cl(\Ltwo,\zbHone)}$. By \cite[Theorem~6.4]{kai97} and Theorem~\ref{thm:V-Vh}, there exists a constant $C_r$ such that
\ben
\NLtwo{n[V] - n[V_h]} \le C_r \NLtwo{V-V_h}\le C_V h.
\een
Moreover, from Lemma~\ref{lem:n-nh}, we easily know
$\NLtwo{n[V_h] - n_h[V_h]} \le C_rh$. The proof is finished.
\end{proof}
\vspace{1mm}

\begin{remark}
The first-order error estimate $\NHone{V-V_h} =O(h)$ is optimal upon using the linear finite element space $X_h$. Therefore, the estimate of $\NLtwo{n[V]-n_h[V_h]}$ is also first-order, as given theoretically in Lemma~\ref{lem:n-nh}. However, numerical experiments in section~\ref{sec:num} imply $\NLtwo{n[V]-n_h[V_h]} =O(h^2)$.
we are going to improve the estimate of $\NLtwo{n[V]-n_h[V_h]}$ by investigating $\NLtwo{V-V_h}$ in a future work.
\end{remark}
\vspace{1mm}

\begin{remark}
Although the error estimate is established for linear finite elements, it is straightforward to extend the result to high-order finite elements. High-order error estimates can also be obtained if the exact solution $V$ is smooth enough.
\end{remark}

\section{Numerical experiments}
\label{sec:num}	

In this section, we report two numerical experiments to verify the finite element error estimates in Section~\ref{sec:fem}. Our code is based on the finite element toolbox PHG which is developed in the State Key Laboratory of Scientific and Engineering Computing (LSEC), Chinese Academy of Sciences. The numerical experiments were carried out on the super computer LSSC-IV of LSEC.

Since the discrete problem \eqref{eq:prob-Vh} is nonlinear, we use the fixed-point iterative algorithm for solving it. Given an approximate solution $V_h^{(k)}\in X_h$,
we compute $V_h^{(k+1)}\in X_h$ with the equation
\begin{align*}
\big(\nabla V_h^{(k+1)},\nabla v_h\big) =
\big(n_h[V_h^{(k)}]-n_D, v_h\big)\qquad \forall\, v_h\in X_h,
\end{align*}
where $n_h[V_h^{(k)}]$ is the approximate electron density defined in \eqref{eq:nh}.
The computational domain is set by $\Omega =(0,1)^3$. Throughout this section, the distribution function is chosen as $f(t)=f_0e^{-\mu t}$ and the number of electrons is given by $N_0=100$.

\begin{example}\label{exa1}
Let $V_0= \sin(\pi x)\sin(\pi y)\sin(\pi z)$ and define the doping file as
\begin{align}\label{exa1-nD}
n_D = \sum_{l=1}^\infty f(\lambda_l-\varepsilon_F)\phi_l^2 -\Delta V_0,
\end{align}
where $\lambda_l,\phi_l$ denote the eigenvalues and eigenfunctions of $(-\Delta)|_{\zbHone}$, respectively, and
the Fermi level $\varepsilon_F$ is determined through the equation $N_0=\sum_{l=1}^\infty f(\lambda_l-\varepsilon_F)$.
Note that the eigenvalues and eigenfunctions of $(-\Delta)|_{\zbHone}$ are given explicitly by
\ben
(i^2+j^2+k^2)\pi^2, \quad \sin(i\pi x)\sin(j\pi y)\sin(k\pi z),\quad
i,j,k\in\bbZ_+.
\een
The exact solution of the Schr\"{o}dinger-Poisson model is given by $V=V_0$ and $(\varepsilon_l,\psi_l) = (\lambda_l,\phi_l)$.
\end{example}

To show the convergence orders of numerical solutions, we designate the number of elements by $N_e$, the $L^2$-error of the numerical potential by $e_{V,0}:=\NLtwo{V-V_h}$,
the $H^1$-error of the numerical potential by $e_{V,1}:=\NHone{V-V_h}$, and the $L^2$-error of the numerical density by $e_{n,0}:=\NLtwo{n-n_h}$. Here $n = \sum_{l=1}^\infty f(\lambda_l-\varepsilon_F)\phi_l^2$ is computed approximately with a truncation of the series such that the truncation error is less than $10^{-8}\times\NLtwo{n}$.

\begin{table}[htp]
\begin{center}
\caption{Convergence orders of $V_h$ and $n_h$ for $\mu=0.1$ and $f_0=1$ $\mathrm{(Example~\ref{exa1})}$.}
\label{tab:exa1-1}
\begin{tabular}{|c|c|c|c|c|c|c|}
\hline
$N_e$ & $e_{V,0}$ & order & $e_{V,1}$ & order & $e_{n,0}$ & order\\
\hline
   	 			$6144$ & $1.85e-02$  & --- & $2.30e-01$ & ---  & $2.82e-02$ & --- \\
   	 			$49152$ & $4.78e-03$ & 1.95 &  $1.09e-01$ & 1.08 & $7.26e-03$ & 1.96 \\
   	 			$393216$ & $1.21e-03$ & 1.99 &    $5.37e-02$ & 1.02 & $1.83e-03$ & 1.99 \\
   	 			$3145728$ & $3.02e-04$ & 2.00 & $2.67e-02$ &  1.01 & $4.58e-04$ & 2.00 \\
\hline
\end{tabular}
\end{center}
\end{table}

\begin{table}[htp]
\begin{center}
\caption{Convergence orders of $V_h$ and $n_h$ for $\mu=2.2\times 10^{-3}$ and $f_0=4.4\times 10^{-6}$ $\mathrm{(Example~\ref{exa1})}$.}
\label{tab:exa1-2}
\begin{tabular}{|c|c|c|c|c|c|c|}
\hline
$N_e$ & $e_{V,0}$ & order & $e_{V,1}$ & order & $e_{n,0}$ & order\\
\hline
$6144$    & $1.77e-02$ & ---     & $2.28e-01$ & ---     & 7.65e-04 & --- \\
$49152$   & $4.55e-03$ & $1.96$ & $1.09e-01$ & $1.07$ & 1.57e-04 & 2.28 \\
$393216$  & $1.15e-03$ & $1.99$ & $5.36e-02$ & $1.01$ & 3.68e-05 & 2.09 \\
$3145728$ & $2.87e-04$ & $2.00$ & $2.67e-02$ & $1.01$ & 9.06e-06 & 2.02 \\
\hline
\end{tabular}
\end{center}
\end{table}

In Table~\ref{tab:exa1-1}, $\mu=0.1$ and $f_0=1$ are set in the distribution function. Clearly we observe the asymptotic behaviors of the errors
\begin{align*}
e_{V,0} = O(h^2), \qquad e_{V,1} = O(h),\qquad e_{n,0} = O(h^2).
\end{align*}
In Table~\ref{tab:exa1-2}, we set $\mu=2.2\times 10^{-3}$ and $f_0=4.4\times 10^{-6}$. For this value of $\mu$, the distribution function $f$ decays much slower than the former case, and we need more terms in the truncation of the series. We still observe the first-order convergence of $\NHone{V-V_h}$ and the second-order convergence of $\NLtwo{V-V_h}$. However, due to the inferior accuracy from approximating large eigenvalues, the convergence order of $\NLtwo{n-n_h}$ is damnified on coarse meshes. The last column of Table~\ref{tab:exa1-2} shows a convergence order even less than $1$ on the mesh with 49,152 elements. But the second-order convergence of $n_h$ is restored on the finest mesh. Here we choose a small value of $f_0$ to guarantee the convergence of the fixed-point iterative algorithm.

\begin{example}\label{exa2}
The applied potential is chosen as $V_0=[e^{x(1-x)}-1][e^{y(1-y)}-1][e^{z(1-z)}-1]$. The definition of $n_D$ is the same as \eqref{exa1-nD}.
\end{example}

In Table~\ref{tab:exa2-1}, we set $\mu=0.1$ and $f_0=1$. Again we observe that the convergence of $V_h$ is of second-order in the $L^2$-norm and of first-order in the $H^1$-norm, and that the convergence of $n_h$ is of second-order in the $L^2$-norm, asymptotically.

\begin{table}[htp]
\begin{center}
\caption{Convergence orders of $V_h$ and $n_h$ for $\mu=0.1$ and $f_0=1$ $\mathrm{(Example~\ref{exa2})}$.}
\label{tab:exa2-1}
\begin{tabular}{|c|c|c|c|c|c|c|}
\hline
$N_e$ & $e_{V,0}$ & order & $e_{V,1}$ & order & $e_{n,0}$ & order\\
\hline
$6144$    & $1.28e-03$ & ---    & $8.75e-03$ & ---     &2.87e-02  & --- \\
$49152$   & $3.38e-04$ & $1.92$ & $3.17e-03$ & $1.47$  &7.41e-03  & 1.95 \\
$393216$  & $8.58e-05$ & $1.98$ & $1.37e-03$ & $1.21$  &1.87e-03  & 1.99 \\
$3145728$ & $2.15e-05$ & $2.00$ & $6.55e-04$ & $1.06$  &4.68e-04  & 2.00 \\
\hline
\end{tabular}
\end{center}
\end{table}

\section{Conclusions}

In this paper, we propose a unified approximation theory for nonlinear problems. Based on the conditions that the approximate operator $\Ca_h$ converges locally to the original operator $\Ca$, and that  $(\Ci-\Ca'[u])^{-1}$ is bounded in the neighborhood of the exact solution $u$, we have obtained the existence and error estimates of the approximate solution $u_h$. Using this theory, we established the optimal finite element error estimate for the numerical solution. The theory has the potential to be used to more general problems. In the next work, we are going to apply
the theory to the finite element approximation of the quantum-corrected drift-diffusion model \cite{ben98,de05}.

\bibliographystyle{plain}

%%%

\end{document}